\documentclass[12pt,reqno]{amsart}

\usepackage[ansinew]{inputenc}
\usepackage{graphicx}
\usepackage[bookmarksnumbered,colorlinks, linkcolor=blue, citecolor=red, pagebackref, bookmarks, breaklinks]{hyperref}

\newtheorem{thm}{Theorem}[section]
\newtheorem{cor}[thm]{Corollary}
\newtheorem{lem}[thm]{Lemma}
\newtheorem{prop}[thm]{Proposition}
\newtheorem*{prop*}{Proposition}

\theoremstyle{definition}
\newtheorem{defi}[thm]{Definition}

\theoremstyle{remark}
\newtheorem{rem}[thm]{Remark}
\numberwithin{equation}{section}

\def\dist{\mathop{\text{\normalfont dist}}}

\usepackage[left=2.9cm,right=2.9cm,top=2.9cm,bottom=2.9cm]{geometry}

\begin{document}

\title[Hopf's lemmas and boundary behaviour of solutions]{Hopf's lemmas and boundary behaviour of solutions to the fractional Laplacian in Orlicz-Sobolev spaces}

\author{Pablo Ochoa}

\address{P. Ochoa. \newline Universidad Nacional de Cuyo, Fac. de Ingenier\'ia. CONICET. Universidad J. A. Maza\\Parque Gral. San Mart\'in 5500\\
Mendoza, Argentina.}
\email{pablo.ochoa@ingenieria.uncuyo.edu.ar}

\author{Ariel Salort}
\address{A. Salort \newline Universidad CEU San Pablo, Urbanización Montepríncipe,
s.n. 28668, Madrid, Spain}
\email{\tt amsalort@gmail.com, ariel.salort@ceu.es}

\parskip 3pt

\maketitle

\begin{abstract}
In this article we study different extensions of the celebrated Hopf's boundary lemma within the context of a  family of nonlocal, nonlinear and nonstandard growth operators. More precisely, we examine the behavior of solutions of the fractional $a-$Laplacian operator near the boundary of a domain satisfying  the interior ball condition. Our approach addresses problems  involving both constant-sign and sign-changing potentials.
\end{abstract}
 
\section{Introduction}

The Hopf's lemma for nonlocal operators has received significant attention in recent years, as it provides a powerful framework for analyzing nonlocal problems. This lemma enables applications across a wide range of scientific fields, including game theory, finance, image processing and Lévy processes among others. See for example \cite{B96, C12, CT, GO} and the references therein.

Classically, the Hopf's lemma provides a refined analysis of the outer normal derivative of superharmonic functions at a minimum boundary point of a bounded domain $\Omega$ that satisfies the interior ball
condition: if $u\in C^2(\overline\Omega)$ is such that $u(x_0)<u(x)$ for all $x\in \Omega$, then $-\Delta u \geq c(x) u $ in $\Omega$ implies that $\frac{\partial u}{\partial \eta }(x_0)<0$. Here $c\in L^\infty(\Omega)$ is a function such that $c(x)\leq 0$ and $\partial u/\partial \eta$ is the outer normal derivative of $u$ at $x_0$. See for instance \cite[Lemma 3.4]{GT}. 

Under the same assumptions on $c$ and $\Omega$, a nonlocal extension of this result was introduced in \cite{ FJ,GS} for the fractional Laplacian. In this generalization, it is shown that if $u$ is a (weak) solution to
\begin{align*}
\begin{cases}
(-\Delta)^s u\geq c u &\text{ in }\Omega,\\  u \geq 0 &\text{ in } \mathbb{R}^n\setminus \Omega
\end{cases}
\quad \text { then } \quad \lim_{B_R\ni x\to x_0 }\frac{u(x)}{\delta_R^s (x)} >0,
\end{align*}
where $\delta_R$ is the distance function from $x$ to $\partial B_R$, being $B_R$ an interior ball at $x_0\in \partial\Omega$.

\noindent Subsequently, in  \cite{DPQ} and later with an alternative proof in \cite{OS}, the previous result was extended to the fractional $p-$Laplacian. This operator, up to a  normalization constant, is defined as
$$
(-\Delta_p)^s:=\text{p.v.} \int_{\mathbb{R}^n} \frac{|u(x)-u(y)|^{p-2}(u(x)-u(y))}{|x-y|^{n+sp}}\,dy
$$
for $p>1$ and $s\in (0,1)$. Under the same assumptions on $\Omega$ and $c$ as  before, here the Hopf's lemma asserts that if $u$ is a (weak) solution to
\begin{align*}
\begin{cases}
(-\Delta_p)^s u\geq c |u|^{p-2}u &\text{ in }\Omega,\\  u \geq 0 &\text{ in } \mathbb{R}^n\setminus \Omega
\end{cases}
\quad \text { then  } \quad \lim_{B_R\ni x\to x_0 }\frac{u(x)}{\delta_R^s (x)} >0.
\end{align*}

\noindent Recently, this result was  extended to cover  the nonlocal non-standard growth operator known as fractional $a-$Laplacian, which is defined in \cite{FBS} as
$$
(-\Delta_a)^s u(x):= \text{p.v.} \int_{\mathbb{R}^n} a\left( |D^s u(x,y)|\right) \frac{D^s u(x,y)}{|D^s u(x,y)|}\frac{dy}{|x-y|^{n+s}}
$$
where $s\in (0,1)$ and $D^s u(x,y):=\frac{u(x)-u(y)}{|x-y|^s}$ is the nonlocal gradient. Here $A$ is a so-called Young function, that is, mainly a convex function $A:[0,\infty)\to [0,\infty)$ such  that $A'(t)=a(t)$ for $t>0$. More precisely, under the previous hypotheses on $\Omega$ and $c$,  and assuming the following growth condition  on $A$:
\begin{align} \label{cond.sen}
1<p\leq \frac{ta'(t)}{a(t)}\leq q<\infty \qquad \text{ for some } p,q>1, 
\end{align}
in \cite{Sen} it is claimed that if $u$ is a weak solution to 
\begin{align} \label{hopf.sen}
\begin{cases}
(-\Delta_a)^s u\geq c a(u) &\text{ in }\Omega,\\  u \geq 0 &\text{ in } \mathbb{R}^n\setminus \Omega
\end{cases}
\quad \text { then } \quad \lim_{B_R\ni x\to x_0 }\frac{u(x)}{\delta_R^s (x)} >0,
\end{align}

As recently discussed in detail in \cite[\text{p.\,2}]{Ataei}, we also remark that we could not verify the arguments presented in \cite[Lemma  4.1]{DPQ}, which are employed to establish Hopf’s principle for the fractional p-Laplacian. Similar reasoning was subsequently employed in \cite{Sen} to derive the Hopf’s principle for the fractional  $a-$Laplacian operator stated in \eqref{hopf.sen},  however, we are unable to fully confirm the validity of this proof as well.

In the article \cite{OS}, we introduced a different approach to that in \cite{DPQ} to establish the Hopf's lemma for the fractional $p-$Laplacian. Building on these ideas, we now present an alternative method to establish a Hopf's lemma for the fractional $a-$Laplacian. This approach is based on constructing a barrier function that involves a scaling of the distance function. However, although this proof overcomes the issue mentioned in \cite[\text{p.\,2}]{Ataei}, it does not yield a lower bound with the expected exponent in the distance function. 

Along this paper we  assume that $A$ is a Young function satisfying the following growth condition:

For our first result we assume  the following growth condition on the Young function $A\in C^2([0,\infty))$: there are $p$ and $q$ such that
\begin{equation}\label{cond.intro1}
0 < p-2 \leq \frac{t a''(t)}{a'(t)} \leq q - 2 <\infty.
\end{equation}By integration by parts twice, \eqref{cond.intro1} yields
\begin{equation} \label{cond.intro}
2 < p \leq \frac{t a(t)}{A(t)} \leq q <\infty.
\end{equation}
Condition \eqref{cond.intro} in the case of the fractional $p-$Laplacian (i.e. $A(t)=t^p$) means that $p>2$. 

In our first result, for continuous supersolutions of the fractional $a-$Laplacian we establish a uniform bound near the boundary expressed in terms of the distance function:

\begin{thm} \label{teo2.intro}
Given $\varepsilon>0$, let $u \in W^{s, A}_0(\Omega)\cap C(\overline \Omega)$ be a weak solution to 
\begin{align} \label{eq.u} 
\begin{cases}
(-\Delta_{a})^s u\geq \varepsilon >0 & \text{ in } \Omega\\
u=0 & \text{ in } \Omega^c,
\end{cases}
\end{align}
where $\Omega$ satisfies the interior sphere condition. Then, for $\rho\in (0,\tfrac12)$ small enough,
$$
u(x)\geq C_\varepsilon d_\Omega(x) \qquad x\in \Omega_\rho=\{x\in\Omega\colon d_\Omega(x)<\rho\}
$$
where $C_\varepsilon$ is a positive constant depending of $\varepsilon$.
\end{thm}

When the equation involves a constant-sign potential function, we establish the following result concerning the boundary behavior of super-solutions:
\begin{thm} \label{teo1.intro}
Let $\Omega\subset \mathbb R^n$ be an open and  bounded  set satisfying the interior ball condition at $x_0\in\partial\Omega$, let $c\in C(\overline \Omega)$ be such that $c(x)\leq 0$ in $\Omega$ and let $u\in \widetilde W^{s,A}(\Omega)\cap C(\overline \Omega)$ a solution to
\begin{align} \label{eq11}
\begin{cases}
(-\Delta_a)^s u\geq c(x)a(u) &\text{ in }\Omega\\
u>0 &\text{ in } \Omega\\
u\geq 0 &\text{ in } \mathbb R^n\setminus \Omega
\end{cases}
\end{align}
in the weak sense. Then
$$
\liminf_{B_R\ni x\to x_0}\frac{u(x)}{\delta_R(x)}>0
$$
where $B_R\subseteq\Omega$ and $x_0\in\partial B_R$ and $\delta_R(x)$ is the distance from $x$ to $B_R^c$.
\end{thm}

Finally, as a counterpart to Theorem \ref{teo1.intro}, we investigate the validity of Hopf's boundary lemma without imposing any sign conditions on the function $c$. Our approach is based primarily on adapting \cite[Theorem 3.1]{OS} and \cite{DSV}, with careful modifications to account for the non-homogeneity of the operator.

Given  a point $x_0\in \partial \Omega$ where the interior ball condition is satisfied, we consider the class $\mathcal{Z}_{x_0}$ consisting of functions $u: \mathbb{R}^{n}\to \mathbb{R}$ such that $u$ is continuous in $\mathbb{R}^{n}$, $|u|>0$  in $B_r(x_r)$ for all sufficiently small $r$, $u(x_0)=0$ and the following growth condition is true
\begin{equation*}
\limsup_{r\to 0}\Phi(r)=+\infty,\quad \text{ where }\quad \Phi(r):= \dfrac{(\inf_{B_{r/2}(x_r)}|u|)^{p-1}}{r^{p s}},
\end{equation*}
where $p$ is given in \eqref{cond.intro}.

For this class of functions we have the following:

\begin{thm}\label{teo3.intro}
Let $\Omega \subset \mathbb{R}^n$ be open and bounded and $x_0 \in \partial \Omega$. Assume that $\Omega$ satisfies the interior ball condition at $x_0$.
Let $u\in \mathcal{Z}_{x_0}$, such that $u^{-}\in L^{\infty}(\mathbb{R}^{n})$, and
$$
(-\Delta_a)^{s}u \geq c(x)a(u)  \quad \text{ in }\Omega$$
weakly, where $c \in L^{1}_{loc}(\Omega)$ with $c^{-}\in L^{\infty}(\Omega)$.
Further, suppose that there is $R>0$ such that $u \geq 0$ in $B_R(x_0)$ and $u > 0 $ in $B_R(x_0)\cap \Omega$. Then, for every $\beta \in (0, \pi/2)$, the following strict inequality holds
\begin{equation}\label{normal derivative at boundary}
\liminf_{\Omega\ni x\to x_0}\dfrac{u(x)-u(x_0)}{|x-x_0|} > 0,
\end{equation}whenever the angle between $x-x_0$ and the vector joining $x_0$ and the center of the interior ball is smaller than $\pi/2-\beta$. 
\end{thm}

A core element  of our proofs is the calculation of the fractional $a-$Laplacian of the  scaled distance function. Unlike the case of powers (i.e. the fractional $p-$Laplacian),  the lack of homogeneity of the operator prevents us  from using  \eqref{cond.intro} to conclude that 
$$
(-\Delta_a)^s(c d^s(x)) \leq C (-\Delta_a)^s u(x)
$$
for some constant depending on $c$, $p$ and $q$. The issue stems from the fact that the sign of $u(x)-u(y)$ is not constant for different values of $x$ and $y$. See section \ref{sec4} for a discussion of this limitation. However, since our barrier arguments only require  $(-\Delta_a)^s(c u(x))$ to be sufficiently small when $c$ is small, it is enough for our purposes to examine certain continuity properties of the operator. The key proposition in our paper is stated in Proposition \ref{prop.cont}:

\begin{prop*}
Let $A$ be a Young function satisfying \eqref{cond.intro1} with $p>\max\{\frac{1}{1-s},2\}$, and let $\Omega\subset \mathbb{R}^n$ be a bounded domain. Let $u\in Lip(\mathbb{R}^n)\cap L^\infty(\mathbb{R}^n)$. Then $(-\Delta_a )^s u$ is continuous at 0 in the sense that if $\{c_k\}_k \subset \mathbb{R}$ is a sequence such that $\lim_{k\to\infty} c_k = 0$ then
$$
\lim_{k\to\infty} (-\Delta_a)^s (c_k u(x)) =0
$$
uniformly for all $x\in \overline \Omega$.
\end{prop*}

The paper is organized as follows. In  Section \ref{preliminares}, we introduce the definition of Young functions, their main properties and the  framework of Orlicz-Sobolev spaces. Next, in Section \ref{preliminary results} we discuss results regarding weak solutions of the fractional $a-$Laplacian.  Continuity properties of $(-\Delta_a)^s$ are provided in Section \ref{sec4}, while our main contributions Theorem \ref{teo1.intro}, Theorem \ref{teo2.intro} and Theorem \ref{teo3.intro} are discussed in Sections \ref{sec5}, \ref{sec6} and \ref{sec7}, respectively.  

\section{Preliminaries}\label{preliminares}

In this section, we give some basic definitions and results related to Orlicz and Orlicz-Sobolev  spaces. We start recalling the definition of an Young.

\subsection{Young functions}
\begin{defi}\label{d2.1}
A function $A \colon [0, \infty) \rightarrow \mathbb{R}$ is called a Young function if it admits the representation
$$A(t)= \int _{0} ^{t} a(\tau) d\tau,$$
where the function $a$ is right-continuous for $t \geq 0$,  positive for $t >0$, non-decreasing and satisfies the conditions
$$a(0)=0, \quad a(\infty)=\lim_{t \to \infty}a(t)=\infty.$$
\end{defi}
\noindent By \cite[Chapter 1]{KR}, a Young function has also the following properties:
\begin{enumerate}
\item[(i)] $A$ is continuous, convex, increasing, even and $A(0) = 0$.
\item[(ii)] $A$ is super-linear at zero and at infinite, that is $$\lim_{x\rightarrow 0} \dfrac{A(x)}{x}=0$$and
$$\lim_{x\rightarrow \infty} \dfrac{A(x)}{x}=\infty.$$
\end{enumerate}

\noindent We will assume that $A$ is extended as a odd function to the whole $\mathbb{R}$.
		
An important property for Young functions is the following:
\begin{defi}

A Young function  $A$ satisfies the $\bigtriangleup_{2}$ condition if  there exists $C > 2$ such that
\begin{equation*}
A(2x) \leq C A(x) \,\,\text{~~~for all~~} x \in \mathbb{R}_+.
\end{equation*}
\end{defi}

\noindent By \cite[Theorem 4.1, Chapter 1]{KR},  a Young function  satisfies $\bigtriangleup_{2}$ condition if and only if there is $q > 1$ such that
\begin{equation}\label{eq p mas}
\frac{ta(t)}{A(t)} \leq q, ~~~~~\forall\, t>0.
\end{equation}

\noindent  Associated to $A$ is  the Young function  complementary to it which is defined as follows:
\begin{equation}\label{Gcomp}
\widetilde{A} (t) := \sup \left\lbrace tw-A(w) \colon w>0 \right\rbrace .
\end{equation}

The definition of the complementary function assures that the following Young-type inequality holds
\begin{equation}\label{2.5}
st \leq A(t)+\widetilde{A} (s) \text{  for every } s,t \geq 0.
\end{equation}

By \cite[Theorem 4.3,   Chapter 1]{KR}, a necessary and sufficient condition for the Young function $\widetilde{A} $ complementary to $A$ to satisfy the $\bigtriangleup_{2}$ condition is that there is $p > 1$ such that
\begin{equation}
p \leq 	\frac{ta(t)}{A(t)}, ~~~~~\forall\, t>0.
\end{equation}
	
From now on, we will assume that the Young function  $A(t)= \int _{0} ^{t} a(\tau) d\tau$  satisfies the following growth behavior:
\begin{equation}\label{G1}
	0 < p-2 \leq \frac{ta''(t)}{a'(t)} \leq q-2 < \infty, ~~~~~\forall t>0.
\end{equation}which, in particular, gives that $A$ and $\tilde{A}$ satisfy the $\bigtriangleup_2$ condition. 

Relation \eqref{G1} gives a lower and upper control of $A$ and its derivatives  in terms of power functions:

\begin{align} \label{potencias}
\begin{split}
\min\{t^p, t^q\}&\leq A(t) \leq \max\{t^p, t^q\},\\
p\min\{t^{p-1}, t^{q-1}\}&\leq a(t) \leq q \max\{t^{p-1}, t^{q-1}\},\\
p(p-1)\min\{t^{p-2}, t^{q-2}\}&\leq a'(t) \leq q(q-1) \max\{t^{p-2}, t^{q-2}\}.
\end{split}
\end{align}
As a consequence $a'(t)>0$ for $t>0$, which yields by \eqref{G1} that $a''(t)>0$. Thus,  $a'(t)$ is strictly increasing for all $t>0$.
\normalcolor

The following lemma is proved in \cite[Lemma 2.1]{MSV}. 
\begin{lem} \label{lema1}
There exists $C>0$ depending only of $p$ and $q$ such that
$$
a(t_2)-a(t_1)\geq C a(t_2-t_1)
$$
for all $t_2\geq t_2$.
\end{lem}

We end the section with the following lemma
\begin{lem} \label{desig}
Let $A$ be a Young function satisfying \eqref{G1}, and $r,t\in \mathbb{R}$, then
$$
|a(r+t)-a(t)| \leq q(q-1) \max\{ (|r|+|t|)^{p-2}, (|r|+|t|)^{q-2}\}|r|.
$$
\end{lem}
\begin{proof}
Since $a'$ is increasing, using \eqref{potencias}
\begin{align*}
a(r+t)-a(t) &= \int_{t}^{r+t} a'(t)\,dt \leq |a'(r+t)||r|\\
& \le q(q-1) \max\{ (|r|+|t|)^{p-2}, (|r|+|t|)^{q-2}\}|r|,
\end{align*}
which concludes the proof.
\end{proof}

\subsection{Orlicz-fractional Sobolev spaces}
Given a Young function $A$ with $a=A'$, $s \in (0, 1)$, 
and an open set $\Omega \subseteq \mathbb{R}^{n}$, 
we consider the spaces:
\begin{align*}
    & L^{A}(\Omega) =
\Big \lbrace u: \Omega \to \mathbb{R} \colon 
    \Phi_{A, \Omega}(u) < \infty \Big \rbrace,\\[5pt]
    &W^{s, A}(\Omega)= 
\Big\lbrace u \in L^{G}(\Omega)\colon 
    \Phi_{s, A, \Omega}(u) < \infty \Big \rbrace, \text{ and }\\[5pt]
  &L_a(\mathbb{R}^n)=\Big\lbrace u \in 
    L^{1}_{loc}(\mathbb{R}^{n})\colon 
    \int_{\mathbb{R}^{n}}a \Big ( \dfrac{|u(x)|}{1+|x|^{s}} \Big ) \dfrac{dx}
    {1+|x|^{n+s}} < \infty\Big\rbrace.
\end{align*}

    Here, the modulars $\Phi_{A, \Omega}$ and $\Phi_{s, A, \Omega}$ are defined as
\[    
\Phi_{A, \Omega}(u) := \int_{\Omega}A(|u(x)|)\,dx, 
\quad\text{ and }\quad
\Phi_{s, A, \Omega}(u) :=
\int_{\Omega\times \Omega}
A \left ( |D^su| \right ) 
    d\nu
\]
being $D^s u = \frac{u(x)-u(y)}{|x-y|^s}$ and $d\nu=\frac{dxdy}{|x-y|^n}$.
Finally, we introduce the space
$$
\widetilde W^{s,A}(\Omega)=\left\lbrace u \in L_a(\mathbb{R}^n): \exists U, \Omega \subset \subset U \text{ and } \|u\|_{W^{s, A}(U)}<\infty\right\rbrace.
$$
\noindent The spaces $L^A(\Omega)$ and $W^{s,A}(\Omega)$ are endowed, 
    respectively, with the following norms
    \[
        \|u\|_{L^A(\Omega)} =
        \inf\left\{\lambda >0 \colon 
        \Phi_{A,\Omega} \left ( \frac{u}{\lambda} \right ) \leq 1\right\},
    \]
    and 
    \[
         \|u\|_{W^{s,A}(\Omega)} = 
         ||u||_{L^A(\Omega)} + [u]_{W^{s,A}(\Omega)},
    \]
    where 
    \[
        [u]_{W^{s,A}(\Omega)} = \inf\left\{\lambda>0 \colon
        \Phi_{s,A,\Omega} \left ( \frac{u}{\lambda} \right ) \leq 1 \right\}.
    \]
    
\noindent  The space $W^{s,A}_0(\Omega)$ is defined as the set of $u\in W^{s, A}(\mathbb{R}^n)$ such that $u=0$ in $\mathbb{R}^n\setminus \Omega$.  Finally,  we recall that under the $\bigtriangleup_2$ condition for $A$ and $\tilde{A}$, all the Orlicz-Sobolev spaces are separable and reflexive spaces.

The following lemma relates modulars and norms in Orlicz spaces.  

\begin{lem}\label{comp norm modular}
Let $A$ be a Young function satisfying \eqref{G1}, and let 
$\xi^\pm\colon[0,\infty)$ $\to\mathbb{R}$ be defined as
\[
\xi^{-}(t)=
\min \big \{  t^{p}, t^{q} \big  \} ,
\quad \text{ and }  \quad
\xi^{+}(t)=\max \big \{  t^{p}, t^{q} \big \} . 
\] 
Then
\begin{itemize}
\item[(i)]
$\xi^{-}(\|u\|_{L^{A}(\Omega)}) \leq \Phi_{A, \Omega}(u) 
\leq  \xi^{+}(\|u\|_{L^{A}(\Omega)})$;

\item[(ii)]
$\xi^{-}([u]_{W^{s, A}(\Omega)}) 
\leq \Phi_{s, A, \Omega}(u) \leq  
\xi^{+}([u]_{W^{s, A}(\Omega)}).$
\end{itemize}
\end{lem}

\subsection{The fractional $a-$Laplacian}
Given a Young function $A$, we define the  fractional $a-$Laplacian of order $s\in (0,1)$ of a sufficiently smooth function $u\colon \mathbb{R}^n\to\mathbb{R}^n$ as
$$
(-\Delta_a)^s u(x) := \text{p.v.} \int_{\mathbb{R}^n} a\left( D^s u(x,y) \right) \frac{dy}{|x-y|^{n+s}}, \quad x\in\mathbb{R}^n,
$$
where $D^s u(x,y):=\frac{u(x)-u(y)}{|x-y|^s}$. Moreover, the following representation formula holds
\begin{equation} \label{rep.form}
\langle (-\Delta_a)^s u, \varphi\rangle:=\dfrac{1}{2}\iint_{\mathbb{R}^n\times\mathbb{R}^n} a\left(D^s u \right) D^s \varphi \frac{dxdy}{|x-y|^n}
\end{equation}
for any   $u, \varphi \in W^{s, A}(\mathbb{R}^n)$. See \cite[Section 6.3]{FBS} for further details.  Moreover, it follows that if $u\in \widetilde W^{s, A}(\Omega)$, then the functional
$$W_0^{s, A}(\Omega) \ni \varphi \mapsto \left\langle u, \varphi\right\rangle,$$is finite, linear and continuous. Hence, the following definition makes sense:
\begin{defi}
We say that $u\in \widetilde W^{s, A}(\Omega)$ is a \emph{weak super-solution} to $(-\Delta_a)^s u =f$ in $\Omega$ if 
$$
\langle (-\Delta_a)^s u, \varphi\rangle \geq \int_\Omega f\varphi\,dx,
$$for any $\varphi\in W^{s, A}_0(\Omega)$, $\varphi\geq 0.$ Similarly, reversing the inequalities we can define \emph{weak sub-solutions} to $(-\Delta_a)^s u=f$.
Finally, $(-\Delta_a)^s u=f$ is satisfied in the \emph{weak sense} if it is simultaneously a weal super- and sub-solution.
\end{defi}

\subsection{The interior ball condition}

Given an open set $\Omega$ of $\mathbb{R}^{n}$ and $x_0 \in \partial \Omega$, we say that $\Omega$ satisfies the  \emph{interior ball condition} at $x_0$ if there is $r_0 >0$ such that, for every $r \in (0, r_0]$, there exists a ball $B_r(x_r) \subset \Omega$ with $x_0 \in \partial \Omega \cap \partial B_r(x_r)$.

By \cite{LY},  $\Omega$ satisfies the interior ball condition if and only if $\partial\Omega\in C^{1,1}$. 

\subsection{Further notation} 

We denote by $\omega_n$ the volume of the unit ball in $\mathbb{R}^n$.

\noindent Given $R>0$, $s>0$ we define the scaled function 
\begin{equation} 
\label{a r} 
a_R(t):=a\left(\frac{t}{R^s} \right).
\end{equation}
Both functions in \eqref{a r} define the same space: denoting $A_R'=a_R$, it holds that
$$
W^{s,A}(\Omega)=W^{s,A_R}(\Omega)
$$
Finally, given $\Omega$ a set in $\mathbb{R}^n$, the distance function $d_\Omega$ is defined as
$$d_\Omega(x):=\dist(x, \Omega^c),$$
where $\dist(x, \Omega^c):=\inf\left\lbrace |x-y|: y\in \Omega^c\right\rbrace$.

\section{Preliminary results on the fractional $a$-Laplacian}\label{preliminary results}

In this section we collect some useful results regarding solutions of the fractional $a-$Laplacian. We start by recalling the following result stated in \cite[Theorem 4.9]{FBSV}.
\begin{prop} \label{prop2}
Let $u\in W^{s,A}_0(\Omega)$ be a weak solution of $|(\Delta_a)^s u |\leq K$ in $\Omega$ for some $K>0$. Then
$$
|u|\leq Cd_\Omega ^s \quad \text{a.e. in }\Omega,
$$
where $C$ is a positive constant depending on $s$, $n$, $A$ and $\Omega$.
\end{prop}

The following \emph{comparison principle} is stated in \cite[Lemma 3.10]{FBPLS}.

\begin{prop} \label{compara}
Let $\Omega\subset \mathbb{R}^n$ be bounded, $u,v\in W^{s,A}(\Omega)\cap C(\Omega)$ such that $u\leq v$ in $\Omega^c$ and assume that
$$
\langle u, \varphi\rangle \leq \langle v,\varphi\rangle \quad \forall \varphi \in W^{s,A}_0(\Omega)\cap C(\Omega), \, \varphi\geq 0 \text{ in }\Omega.
$$
Then $u\leq v$ in $\Omega$.
\end{prop}

A \emph{weak maximum principle} is stated in the following proposition.
\begin{prop}  \label{pdm}
Let $A$ be a Young function such that $p>2$. Let $\Omega\subset \mathbb{R}^n$ be open  and bounded. Then, if $u\in W^{s,A}_0(\Omega)$ is a weak supersolution to $(-\Delta_a)^s u=0$ in $\Omega$, $u=0$ in $\mathbb{R}^n\setminus \Omega$ then $u\geq 0$ in $\mathbb{R}^n$.
\end{prop}

\begin{proof}
Let $u\in W^{s,A}_0(\Omega)$, $u\neq 0$, be a weak supersolution to $(-\Delta_a)^s u=0$ in $\Omega$. Let $u^\pm:=\max\{\pm u,0\}$ be the positive and negative parts of $u$, respectively. Assume $u^{-}\neq 0$.  Testing with $u^-\in W^{s,A}_0(\Omega)$ we obtain that
\begin{align*}
0&\leq 	\iint_{\mathbb{R}^n\times \mathbb{R}^n} a(|D^s u|)\frac{D^s u}{|D^s u|}D^s u^-\,d\nu=\iint_{\mathbb{R}^n\times \mathbb{R}^n}  D^s u D^s u^- \left(\int_0^1 a'((1-t)|D^s u|)\,dt \right)d\nu.
\end{align*}
We observe that
\begin{align} \label{desss}
\begin{split}
D^s u(x,y)D^su^-(x,y)&= \frac{u^+(x)-u^+(y) -(u^-(x)-u^-(y))}{|x-y|^s}\frac{u^-(x)-u^-(y)}{|x-y|^s}\\
&=
 - \frac{u^+(x) u^-(y) + u^+(y)u^-(x)}{|x-y|^{2s}} - (D^su^-(x,y))^2\leq 0.
 \end{split}
\end{align}
Moreover,  $p>2$ yields that $a'(t)\geq 0$ for all $t\geq 0$ and $a'(t)=0$ if and only if $t=0$. This implies that $D^s u^-\equiv 0$ in $\mathbb{R}^n\times \mathbb{R}^n$, that is, $u^- \equiv 0$, which gives $u\geq 0$ in $\mathbb{R}^n$.
\end{proof}

The following relation between weak and viscosity solutions stated in \cite[Lemma 3.7]{FBPLS}. For further details on the definition of viscosity solution for solutions of $(-\Delta_a)^s$ and the equivalence of solutions, see for instance \cite[Section 3]{FBPLS} and \cite{DDO}.

\begin{prop} \label{equiv}
Assume $p>\frac{1}{1-s}$. Let $u\in W^{s,A}_0(\Omega)$  be a  weak solution to $(-\Delta_a)^s u=0$ in $\Omega$, $u=0$ in $\mathbb{R}^n\setminus \Omega$. Then $u$ is a viscosity solution to the same equation.
\end{prop}

With the aid of the previous result we prove the following \emph{strong maximum principle for continuous solutions}.
\begin{prop}  \label{pfm}
Let $u\in W^{s,A}_0(\Omega)\cap C(\Omega)$ be a weak supersolution to $(-\Delta_a)^s u=0$ in $\Omega$, $u=0$ in $\mathbb{R}^n\setminus \Omega$. Then, either $u>0$ in $\Omega$ or $u\equiv 0$ in $\mathbb{R}^n$.
\end{prop}

\begin{proof}
Let $u\in W^{s,A}_0(\Omega)$ be a weak supersolution. By Proposition \ref{pdm} we can assume that $u\geq 0$ in $\Omega$. Using Proposition \ref{equiv} we have that $u$ also is a viscosity supersolution to the same equation. Let $x_0\in \Omega$ be a point where $u(x_0)=0$. Since $u$ is continuous, by definition of viscosity supersolution, for any test function $\varphi\in C_c^1(\mathbb{R}^n)$ such that
$$
0=u(x_0)=\varphi(x_0), \qquad \varphi(x)<u(x) \text{ if } x\neq x_0
$$
it holds that
$$
0\geq -(-\Delta_a)^s\varphi(x_0)= \int_{\mathbb{R}^n} a\left( \frac{|\varphi(y)|}{|x_0-y|^s}  \right)\frac{\varphi(y)}{|\varphi(y)|} \frac{dy}{|x_0-y|^{n+s}}.
$$
So, if $\varphi\geq 0$ then it follows  that $\varphi\equiv 0$, from where $u\equiv 0$ in $\mathbb{R}^n$.
\end{proof}

\section{Continuity properties of the fractional $a$-laplacian} \label{sec4}
  
A central aspect of our argument involves computing the fractional $a-$Laplacian of the scaled distance. For fractional $p-$Laplacian, the homogeneity of the operator plays a key role. Specifically, when $u\in C_c^\infty(\mathbb{R}^n)$ and $
c>0$ is a scaling factor, we have the following property:
$$
(-\Delta_p)^s u(x) = c^p (-\Delta_p)^s u(x).
$$
This identity is essential when employing barrier arguments, as it allows for the preservation of homogeneity under scaling transformations.

\noindent However, when attempting to derive an analogous expression for the fractional $
a-$Laplacian, we find out a significant obstacle: we cannot apply conditions \eqref{G1} and \eqref{potencias} to establish  
$$
(-\Delta_a)^s(c u(x)) \leq c^p (-\Delta_a)^s u(x) \quad \text{ or } \quad 
(-\Delta_a)^s(c u(x)) \leq c^q (-\Delta_a)^s u(x).
$$
The difficulty arises because the sign of the difference $u(x)-u(y)$ is not constant; instead, it depends on the specific points $x$ and $y$. This dependency complicates the application of homogeneity arguments and prevents a straightforward extension of the previous result. To address this issue, in the following proposition we study some  continuity properties of the operator.

\begin{prop} \label{prop.cont}
Let $A$ be a Young function satisfying \eqref{G1} with $p>\max\{\frac{1}{1-s},2\}$ and let $\Omega\subset \mathbb{R}^n$ be a bounded domain. Let $u\in Lip(\mathbb{R}^n)\cap L^\infty(\mathbb{R}^n)$. Then $(-\Delta_a )^s u$ is continuous at 0 in the sense that if $\{c_k\}_k \subset \mathbb{R}$ is a sequence such that $\lim_{k\to\infty} c_k = 0$ then
$$
\lim_{k\to\infty} (-\Delta_a)^s (c_k u(x)) =0
$$
uniformly for all $x\in \overline \Omega$.
\end{prop}
\begin{proof}
 Given $x\in \overline{\Omega}$, for each $k\in \mathbb{N}$ we define the function
$$
f_k(x)= \int_{\mathbb{R}^n} a\left( c_k \frac{u(x)-u(y)}{|x-y|^s} \right)\frac{dy}{|x-y|^{n+s}}.
$$
Changing variables we get the following identity
$$
f_k(x)=\int_{\mathbb{R}^n} a\left( c_k \frac{u(x)-u(x+z)}{|z|^s} \right)\frac{dz}{|z|^{n+s}} 
=\int_{\mathbb{R}^n} a\left( c_k \frac{u(x)-u(x-z)}{|z|^s} \right)\frac{dz}{|z|^{n+s}},
$$
which allow us to write
$$
f_k(x)=\frac12 \int_{\mathbb{R}^n} \left( a(c_k D^s u(x,x+z)) + (c_k D^s u(x,x-z)) \right) \frac{dz}{|z|^{n+s}}.
$$

Along this proof, for $t>0$, we denote $M(t)=\max\{t^{p-2},t^{q-2}\}$. Also, for $u\in Lip(\mathbb{R}^n)$ we denote by $L$ the Lipschitz constant of $u$.

\noindent {\bf  Local  Uniform boundedness:} 
For each $n\in\mathbb{N}$ and $x\in \overline{\Omega}$, we split $f_n(x)$ as
\begin{align*}
f_k(x) &= \frac12 \int_{|z|>1}a\left( c_k D^s u(x,x+z) \right)\frac{dz}{|z|^{n+s}}+\frac12
\int_{|z|<1}a\left( c_k D^s u(x,x+z) \right)\frac{dz}{|z|^{n+s}}\\
&:= I_1(x) + I_2(x).
\end{align*}

\noindent \underline{Bound for $I_1$}. Since $u\in L^\infty(\mathbb{R}^n)$ and $|c_k|\leq K$ for some $K>0$, for any $x\in \overline{\Omega}$,  we get
\begin{align*}
I_1(x)& \leq \frac12 \int_{|z|>1} a(K 2 \|u\|_\infty)\frac{dz}{|z|^{n+s}} = \frac{n \omega_n}{2}a(2K\|u\|_\infty) \int_1^\infty 	r^{-s-1} \,dr= \frac{n \omega_n}{2s}a(2K\|u\|_\infty). 
\end{align*}

\noindent \underline{Bound for $I_2$}.   Observe that using Lemma \ref{desig} and that $u\in Lip(\mathbb{R}^n)$, for any $x\in \overline{\Omega}$ and $z\in\mathbb{R}^n$ such that $|z|<1$ we get 
\begin{align*}
|a&\left( c_k D^s u(x,x+z) \right)| \leq M\left( \frac{K|u(x)-u(x-z)|}{|z|^s}\right) \frac{K|u(x)-u(x+z)|}{|z|^s}\\
&\leq M(K L|z|^{1-s}) K L|z|^{1-s}\\
&\leq C |z|^{(p-1)(1-s)}.
\end{align*}
Therefore, provided $p> 1/(1-s)$, we have that
\begin{align*}
I_2(x)&\leq C \int_{|z|<1}   |z|^{(p-1)(1-s)-n-s}\,dz =  \frac{Cn\omega_n}{2(1-s)}.
\end{align*}

\medskip
 
\noindent {\bf Uniform equicontinuity.} 
 For every $\varepsilon>0$ we want to see that
$|f_n(x_1)-f_n(x_2)|<\varepsilon$  whenever $x_1,x_2\in \overline{\Omega}$ are such that, for $\delta=\delta(\varepsilon)$,  $|x_1-x_2|<\delta$.
 
\noindent For this, we split the difference as 
\begin{align*}
f_k(x_1)-f_k(x_2) = I_1(x_1,x_2) + I_2(x_1,x_2)
\end{align*}
where
\begin{align*}
I_1(x_1,x_2)&:=
 \int_{|x_1-y|>1} a(c_k D^s u(x_1,y))  \frac{dy}{|x_1-y|^{n+s}}
 -
\int_{|x_2-y|>1} a(c_k D^s u(x_2,y))  \frac{dy}{|x_2-y|^{n+s}},
 \\
I_2(x_1,x_2)&:=
\int_{|x_1-y|\leq 1} a(c_k D^s u(x_1,y))  \frac{dy}{|x_1-y|^{n+s}}
 -
\int_{|x_2-y|\leq 1} a(c_k D^s u(x_2,y))  \frac{dy}{|x_2-y|^{n+s}}.
\end{align*}

\noindent  \underline{Bound for $I_1$}.  For any $x_1,x_1\in \mathbb{R}^n$ and $z\in \mathbb{R}^n$ such that $|z|>1$,  using Lemma \ref{desig} we have the following:
\begin{align} \label{cotaI}
\begin{split}
I(&x_1,x_2,z):=a\left(c_k \frac{u(x_1)-u(x_1+z)}{|z|^s}  \right) - a\left(c_k \frac{u(x_2)-u(x_2+z)}{|z|^s} \right) \\
&\leq
M 
\left(\frac{K|u(x_1)-u(x_1+z)|}{|z|^s} + \frac{K|u(x_1)-u(x_1+z)-u(x_2)+u(x_2+z)|}{|z|^s} \right)\\
& \quad \times \frac{K|u(x_1)-u(x_1+z)-u(x_2)+u(x_2+z)|}{|z|^s}.
\end{split}
\end{align}
Since $u\in Lip(\mathbb{R}^n)$ and $|z|>1$, we can bound \eqref{cotaI} as follows: 
\begin{align} \label{estima1}
\begin{split}
I(x_1,x_2,z)&\leq 2KL \dfrac{|x_1-x_2|}{|z|^s} M\left(\frac{6K\|u\|_{\infty}}{|z|^s} \right) \leq 
C \dfrac{|x_1-x_2|}{|z|^{s(q-1)}}, 
\end{split}
\end{align}
where $C$ is a positive constant depending only of $K$, $L$, $\|u\|_\infty$, $p$ and $q$.

Therefore, we get that
\begin{align*}
|I_1(x_1,x_2)| &= \left|
 \int_{|z|>1} \left( a\left(c_k \frac{u(x_1)-u(x_1+z)}{|z|^s}  \right) - a\left(c_k \frac{u(x_2)-u(x_2+z)}{|z|^s} \right)  \right) \frac{dz}{|z|^{n+s}} 
\right|\\
&\leq
C|x_1-x_2|\int_{|z|>1} \frac{dz}{|z|^{n+sq}} =  \frac{n\omega_nC}{sq} |x_1-x_2|.
\end{align*}

\noindent  \underline{Bound for $I_2$}. For any $x_1,x_2\in \overline{\Omega}$ and $z\in\mathbb{R}^n$ such that $|z|\leq 1$, we can bound \eqref{cotaI} as
\begin{align} \label{estima4}
\begin{split}
I(x_1,x_2,z)&\leq M(3K L|z|^{1-s}) 2K L |x_1-x_2| |z|^{1-s}\\ 
&\leq 
C |z|^{(1-s)(p-1)}|x_1-x_2|,
\end{split}
\end{align}
where $C>0$ depends only of $p$, $q$, $K$ and  $L$.

Therefore, by changing variables and using \eqref{estima4}, provided again $p>\frac{1}{1-s}$, we get that
\begin{align*}
|I_2(x_1,x_2)| &= \left|
 \int_{|z|\leq 1} \left( a\left(c_k \frac{u(x_1)-u(x_1+z)}{|z|^s}  \right) - a\left(c_k \frac{u(x_2)-u(x_2+z)}{|z|^s} \right)  \right) \frac{dz}{|z|^{n+s}} 
\right|\\
&\leq C|x_1-x_2|\int_{|z|\leq 1} |z|^{(p-1)(1-s)-n-s}\,dz\\
&=C|x_1-x_2|.
\end{align*}

\medskip

Then, by Arzelá-Ascoli theorem, we get that, up to a subsequence,  $f_k$ converges uniformly to $0$ in $\overline{\Omega}$. This concludes the proof.
\end{proof}

As a straightforward consequence, we get the next corollary.

\begin{cor} \label{coro1}
Assume that $A$ is  a Young function satisfying \eqref{G1} with $p>\max\{\frac{1}{1-s},2\}$. Let $u\in Lip(\mathbb{R}^n)\cap L^\infty(\mathbb{R}^n)$ and let $\Omega\subset \mathbb{R}^n$ be a bounded domain. For any $\varepsilon >0$,  there exists $C>0$  such that for all $c\in (0, C)$ it holds that
$$
|(-\Delta_a)^s(c u(x))|\leq \varepsilon,
$$for all $x\in \overline{\Omega}$. 
\end{cor}

Finally, we state the following key relation between weakly and pointwisely  bounded $a$-Lapalcian.

\begin{prop}\label{p a w}
Let $\Omega\subset \mathbb{R}^n$ be a bounded open domain. Assume that for $\varepsilon>0$,
$$(-\Delta_a)u(x)\leq \varepsilon,$$
pointwisely for any $x\in \Omega$. Then, $(-\Delta_a)u\leq \varepsilon$ in the weak sense in $\Omega$. 
\end{prop} 
\begin{proof}
Let $\varphi\in C_c^{\infty}(\Omega)$ such that $\varphi\geq 0$. Then, according to the representation formula \eqref{rep.form} we derive the following:
\begin{equation}
\begin{split}
\left\langle (-\Delta_a)u, \varphi\right\rangle & = \dfrac{1}{2}\iint_{\mathbb{R}^n \times \mathbb{R}^n} a\left(\dfrac{u(x)-u(y)}{|x-y|^s} \right)\dfrac{\varphi(x)-\varphi(y)}{|x-y|^s}\dfrac{dx\,dy}{|x-y|^{n}}\\& = \dfrac{1}{2}\int_{\mathbb{R}^n}\varphi(x)\left[ \int_{\mathbb{R}^n} a\left(\dfrac{u(x)-u(y)}{|x-y|^s} \right)\dfrac{dy}{|x-y|^{s+n}}\right]\,dx\\& \quad -\dfrac{1}{2}\int_{\mathbb{R}^n}\varphi(y)\left[- \int_{\mathbb{R}^n} a\left(\dfrac{u(y)-u(x)}{|y-x|^s} \right)\dfrac{dx}{|y-x|^{s+n}}\right]\,dy\\& = \dfrac{1}{2}\int_{\mathbb{R}^n}\varphi(x)(-\Delta_a)^s u(x) \,dx+\dfrac{1}{2}\int_{\mathbb{R}^n}\varphi(y)(-\Delta_a)^s u(y) \,dy \\&
\leq \varepsilon \int_{\mathbb{R}^n}\varphi(x)\,dx.
\end{split}
\end{equation}

This ends the proof of the proposition. 
\end{proof}

\begin{rem}Given an open bounded set $\Omega$, we point out that all of the results of this section apply to the distance function $d_\Omega$ since $d_\Omega\in Lip(\mathbb{R}^n)\cap L^{\infty}(\mathbb{R}^{n})$.
\end{rem}

\section{Hopf's lemma for torsion-like problems}\label{sec5}

Let $A$ be a Young function satisfying \eqref{G1} $p>\max\{\frac{1}{1-s},2\}$. For  $R\in (0,1)$, denote by $u_R \in W^{s,A}_0(B_R)\cap C(\overline B_R)$ the unique solution to
\begin{align*}  
\begin{cases}
(-\Delta_{a})^s u_R=\beta & \text{ in } B_R\\
u_R=0 & \text{ in } B_R^c
\end{cases}
\end{align*}
where $\beta>0$ is a fixed constant. Then we have the following proposition.

\begin{prop} \label{pre.hopf}
There exists $C_1,c_2>0$ such that for all $R\in(0,1)$ and $x\in \mathbb{R}^n$,
$$
C_1d_{B_R}(x) \le u_R(x) \leq C_2 d^s_{B_R}(x).
$$

\end{prop}
\begin{proof}

Fix $0<R<1$.

\noindent {\bf Step 1}.
Corollary \ref{coro1} gives the existence of  $\lambda_0\in (0,1)$ such that 
\begin{equation*}
|(-\Delta_{a})^s( \lambda d_{B_1})|  \leq \beta \quad \text{ in } B_1,
\end{equation*}
for any $\lambda\in (0, \lambda_0]$. Let us fix $\lambda \in (0, \lambda_0)$.  Moreover,  by Proposition \ref{p a w}, we have that
\begin{equation}\label{ineq d}
(-\Delta_{a})^s( \lambda d_{B_1})\leq \beta,
\end{equation}
in the weak sense in $ B_1$.

We denote by $u_1\in W^{s,A}_0(B_1)$ the unique weak solution to
\begin{align} \label{eqR} 
\begin{cases}
(-\Delta_{a_R})^s u_1=\beta & \text{ in } B_1\\
u_1=0 & \text{ in } B_1^c,
\end{cases}
\end{align}where we recall that $a_R$ is given by \eqref{a r}. From Proposition \ref{pfm}, it follows that $u_1>0$ in $B_1$. We define the function  $w(x)$ as 
$$w(x)=\lambda R^sd_{B_1}(x).$$

\noindent Then, in the weak sense, and by \eqref{ineq d}, we have that
\begin{align*}
(-\Delta_{a_R})^s w  &= (-\Delta_{a_R})^s (\lambda R^s d_{B_1})\\
&=
(-\Delta_a) (\lambda d_{B_1})\leq \beta \leq (-\Delta_{a_R})^s u_1\text{ in } B_1.
\end{align*}
Moreover, $u_1=d_{B_1} =0$ in $B_1^c$. Hence, by the comparison principle given  in Proposition \ref{compara},
\begin{equation*}
w   = \lambda R^s d_{B_1}  \leq u_1 \quad \text{ in } \mathbb{R}^n.
\end{equation*}

On the other hand, by using Proposition \ref{prop2}, we have that
\begin{equation*}  
u_1 \leq C d_{B_1}^s R^s \quad \text{ in } B_1,
\end{equation*}
and then, for some $C_1,C_2>0$
\begin{equation} \label{desig1}
C_1R^s d_{B_1} \le u_1 \leq C_2 R^s d_{B_1}^s \quad \text{ in } \mathbb{R}^n.
\end{equation}
Since $R<1$, we get $R<R^s$ and so
\begin{equation} \label{desig11}
C_1R d_{B_1} \le u_1 \leq C_2 R^s d_{B_1}^s \quad \text{ in } \mathbb{R}^n.
\end{equation}

\medskip
\noindent {\bf Step 2}. Denote by $u_R\in W^{s,A}_0(B_R)$ the unique weak solution to
\begin{align*}  
\begin{cases}
(-\Delta_{a})^s u_R=\beta & \text{ in } B_R\\
u_R=0 & \text{ in } B_R^c.
\end{cases}
\end{align*}
by \cite[Lemma C.1]{FBSV}, if we define $v(x)=u_R(Rx)$, this scaled function solves
\begin{align} \label{eqR.29} 
\begin{cases}
(-\Delta_{a_R})^s v=\beta & \text{ in } B_1\\
v=0 & \text{ in } B_1^c.
\end{cases}
\end{align}
By uniqueness of solution of \eqref{eqR.29} and \eqref{eqR} we get $v=u_1$. Hence, from \eqref{desig1} we obtain
$$
C_1R  d_{B_1}(x) \le u_R(Rx) \leq C_2 R^s d_{B_1}^s(x), \qquad \text{ for any } x\in \mathbb{R}^n.
$$
Since $d_{B_R}(Rx)=Rd_{B_1}(x)$, 
$$
C_1 d_{B_R}(Rx) \le u_R(Rx) \leq C_2 d_{B_R}^s(Rx), \qquad \text{ for any } x\in \mathbb{R}^n.
$$
This ends the proof. 
 \end{proof}
 
We are now in position to prove Theorem \ref{teo2.intro}.

\begin{proof}[Proof of Theorem \ref{teo2.intro}]
Fix $x\in \Omega_\rho$ for $\rho\in(0,\tfrac12)$. By the interior sphere property there is a ball $B\subset \Omega$ of radius $2\rho<1$, tangent to $\partial\Omega$ such that $d_\Omega(x)=d_B(x)$. Given $0 <\beta\leq \varepsilon$, let $v\in W^{s,A}_0(B)$ solution to 
\begin{align} \label{eqR.2} 
\begin{cases}
(-\Delta_{a})^s v=\beta & \text{ in } B\\
v=0 & \text{ in } B^c.
\end{cases}
\end{align}
So, in $B$ we have that $(-\Delta_a)^s v \le (-\Delta_a)^s u$ in $B$, and $v\leq u$ in $B^c$. By comparison principle, $v\leq u$ in $\mathbb{R}^n$. Since $d_\Omega(x)=d_B(x)$, from Proposition \ref{pre.hopf} we get that
\begin{equation} \label{eqh1}
C_1d_\Omega(x) \le v(x) \leq u(x)
\end{equation}
which concludes the proof.
\end{proof}

\section{Hopf's boundary lemma with constant-sign potentials}\label{sec6}
 
In this section, we  prove a Hopf's lemma for weak solutions of the fractional $a$-laplacian with a constant-sign potential.
  
\begin{proof}[Proof of Theorem \ref{teo1.intro}]

Firstly, observe that if $x_0\in \partial \Omega$ and $u(x_0)>0$, then there is a neighbourhood of $x_0$ in $\overline{\Omega}$ where $u\geq c >0$ for some $c$. Then, the conclusion 
$$
\liminf_{B_R\ni x\to x_0}\frac{u(x)}{\delta_R(x)}>0
$$holds trivially. Thus, we will assume that $u=0$ on $\partial \Omega$. 

For a given $x_0\in \partial \Omega$, by the regularity of $\Omega$, 
there exists $x_1\in \Omega$ on the normal line to $\partial \Omega$ at $x_0$ and $r_0> 0$ such that
$$
B_{r_0}(x_1)\subset \Omega,\quad  \overline{B}_{r_0}(x_1)\cap \partial \Omega= \{x_0\} \quad \text{and }\quad \text{dist} (x_1,\Omega^{c})=|x_1-x_0|.
$$
We will assume without loss of generality that $x_0= 0$, $r_0=1$ and $x_1= e_n$, with $e_n=(0, \ldots, 0, 1)\in \mathbb{R}^{n}$ and consider a nontrivial weak supersolution $u$  to \eqref{eq11}.
 
\medskip
 
Let $d$ be the distance function
 $d\colon \mathbb R^n\to \mathbb R$ given by $d(x)=\dist (x,B_1^c(e_n))$. We will build now a suitable weak supersolution.
 
 From Corollary \ref{coro1}, given $\beta>0$ (to be determined later) there exists $\lambda\in(0,1)$ small enough  such that for all $\lambda_0\in (0,\lambda]$
$$
|(-\Delta_a)^s(\lambda_0 d(x))| \leq \beta, \quad \text{for all }x\in B_1(e_n)\cap B_r(0),
$$
and by Proposition \ref{p a w}, 
$$
(-\Delta_a)^s(\lambda_0 d) \leq \beta \quad \text{ weakly in } B_1(e_n)\cap B_r(0).
$$
 Let $D\subset\subset B_1^c(e_n)\cap \Omega$ be a smooth domain. Define 
$$
\underline u(x)= \lambda d(x)+ \chi_D(x)  u(x).
$$
 
\noindent By \cite[Lemma C.5]{FBSV}, we get
$$
(-\Delta_a)^s \underline u \leq \beta + h \quad \text{weakly in }B_1(e_n)\cap B_r(0),
$$
where the function $h$ is given by
\begin{equation}
h(x)= 2\int_D \left[a\left(\dfrac{\lambda d(x)-\lambda d(y)-\chi_D(y)u(y)}{|x-y|^s} \right)-a\left(\dfrac{\lambda d(x)-\lambda d(y)}{|x-y|^s} \right)\right] \dfrac{dy}{|x-y|^{s+n}}.
\end{equation}
Applying Lemma \ref{lema1} and using the fact that $a$ is odd leads to 
$$
h(x)\leq 2c \int_D a\left(\dfrac{-u(y)}{|x-y|^s} \right) \dfrac{dy}{|x-y|^{s+n}},
$$
where $c>0$ is a constant depending only of $p$ and $q$.

Then, since $D\subset\subset B_1^c(e_n)\cap \Omega$, there is $C_D>1$ such that $|x-y|\leq C_D$ for any $x\in B_1(e_n)$ and $y\in D$. Using that $a$ is increasing and odd we get
\begin{align*}
h(x)&\leq  -2c \int_D a\left(  \frac{u(y)}{|x-y|^s}\right) \frac{dy}{|x-y|^{s+n}} \leq 
-2c \int_D a\left(  \frac{u(y)}{C_D^s}\right) \frac{dy}{C_D^{s+n}}\\ 
&\leq 
-a(M_0)\bar{C}_D|D|:= -\widetilde{M}_0,
\end{align*}
where $\bar C_D:= 2cC_D^{-(n+s(q+1))}$ and  $M_0:=\min_{x\in D} u(x)>0$.
 
 Now, define 
$$
M_1:=\inf_{x\in B_1(e_n)\cap B_r^c(0)} u(x)>0, \qquad 
M_2:=\sup_{x\in B_1(e_n)\cap B_r(0)} u(x)>0, 
$$
and we take $r$ small enough  so that $r \in (0, r_0)$, and
$$
a(M_2) < \dfrac{\widetilde M_0}{\|c\|_{\infty}}.
$$
This can be done since $u$ is continuous in $\overline{\Omega}$ and $u = 0$ on $\partial\Omega$. Observe that this bound is uniform and independent of $r$, although  $r$ depends on the boundary point.
Next, choose 
 
$$
0 < \beta\leq  \widetilde M_0-a(M_2) \|c\|_\infty 
$$
which leads to  (in the weak sense)
\begin{align} \label{eqcomp}
\begin{split}
(-\Delta_a)^s \underline u &\leq \beta   -\widetilde{M}_0 \leq - a(M_2) \|c\|_\infty\\
&\leq c a(u)\leq (-\Delta_a)^s u \qquad \text{ in } B_1(e_n)\cap B_r(0).
\end{split}
\end{align}
Then,  for $x\in B_1^c(e_n)$ we have that
$$
\underline{u}(x) = u(x) \chi_D(x) \leq u(x) 
$$
and for $x\in B_1(e_n)\setminus B_r(0)$ we have, taking $\lambda$ small enough, that
$$
\underline{u}(x) = \lambda d(x) \leq \lambda \leq u(x).
$$
In sum, we have obtained from the previous expression and  \eqref{eqcomp} that
\begin{align*}
\begin{cases}
(-\Delta_a)^s \underline u  \leq (-\Delta_a)^s u &\quad \text{ weakly in } B_1(e_n)\cap B_r(0),\\
\underline u(x)\leq u(x) &\quad \text{ in } (B_1(e_n)\cap B_r(0))^c.
\end{cases}
\end{align*}
Using the comparison principle it  gives that
\begin{equation} \label{desuu}
\underline u(x)\leq u(x) \quad \text{ in } B_1(e_n)\cap B_r(0).
\end{equation}
By definition of $d(x)$, for any $t\in (0,1)$
\begin{equation} \label{Des1}
d(te_n)=\delta(te_n)
\end{equation}
where $\delta(x)=\dist(x,\Omega^c)$, and since $te_n \notin D$
\begin{equation} \label{Des2}
\underline u(te_n) = \beta  d(te_n),
\end{equation}
this gives, from \eqref{Des1}, \eqref{Des2} and \eqref{desuu}, that
$$
\frac{u(te_n)}{\delta(te_n)}=\frac{u(te_n)}{d(te_n)}=\frac{u(te_n)}{\beta^{-1}\underline u(te_n)} \geq \beta>0
$$
which completes the proof.
\end{proof}

\section{Hopf's boundary lemma for sign-changing potentials}\label{sec7}

In this section, we will provide a Hopf's boundary lemma for solutions of
$$
(-\Delta_a)^{s}u \geq c(x)a(u)  \quad \text{ in }\Omega$$where now no sign condition is impose to $c$. We will mainly follow the plan of \cite{DSV} and its nonlinear counterpart \cite[Theorem 3.1]{OS}, with the necessary changes due to the lack of homogeneity of the operator. Hence, we will provide all the details.  We also deduce some interesting consequences.

Following \cite{DSV}, we consider the next class of functions. Given  a point $x_0\in \partial \Omega$ where the interior ball condition holds, the class $\mathcal{Z}_{x_0}$ consists of functions $u: \mathbb{R}^{n}\to \mathbb{R}$ such that $u$ is continuous in $\mathbb{R}^{n}$, $|u|>0$  in $B_r(x_r)$ for all sufficiently small $r$, $u(x_0)=0$ and the following growth condition is true
\begin{equation}\label{growth boundary u}
\limsup_{r\to 0}\Phi(r)=+\infty
\end{equation}where
\begin{equation}\label{defi phi}
\Phi(r):= \dfrac{(\inf_{B_{r/2}(x_r)}|u|)^{p-1}}{r^{p s}},
\end{equation}
where we recall that $p$ satisfies
$$p\leq \dfrac{ta(t)}{A(t)}, \quad \text{for all }t>0.$$

\medskip
We state now the main theorem of this section.
We are now in position to prove Theorem \ref{teo3.intro}

\begin{proof}[Proof of Theorem \ref{teo3.intro}]
For the sake of simplicity we split the proof in four stages.

\textbf{Step 1}: Consider the distance function for the unit ball
$$d(x)=\text{dist}(x, B_1^c)$$and recall, for further reference, that any distance function for bounded domains satisfies Proposition \ref{prop.cont} and Proposition \ref{p a w}. Also,  $d$ satisfies  the lower bound
\begin{equation}\label{lower bound v}
d(x) \geq \dfrac{1}{2}(1-|x|^{2}).
\end{equation}
We will use this estimate in what follows.

\medskip

\textbf{Step 2}: In order to localize the argument near the boundary point $x_0$, we will build appropriate functions to compare with $u$.

Take $\rho\in (0, 1/2)$ and  $\eta \in C_0^{\infty}(B_{1-2\rho})$  nonnegative, $\eta \leq 1$, with $\int_{\mathbb{R}^{n}}\eta =1$. Then, for $x \in B_1\setminus \overline{B_{1-\rho}}$, we have
\begin{equation}\label{g lap eta}
\begin{split}
(-\Delta_a)^{s}\eta(x)& = 
\text{p.v.}\, \int_{\mathbb{R}^{n}}a\left(\dfrac{\eta(x)-\eta(y)}{|x-y|^{s}} \right)\dfrac{dy}{|x-y|^{n+s}}= -\text{p.v.}\,\int_{\mathbb{R}^{n}}a\left(\dfrac{\eta(z+x)}{|z|^{s}} \right)\dfrac{dz}{|z|^{n+s}}\\ & = 
-\text{p.v.}\, \int_{B_{1-2\rho}(-x)}a\left(\dfrac{\eta(z+x)}{|z|^{s}} \right)\dfrac{dz}{|z|^{n+s}}.
\end{split}
\end{equation}
Since $z \in B_{1-2\rho}(-x)$, there holds
$$
|z| \leq |z+x|+|x|\leq 2-2\rho,
$$
from where, \eqref{g lap eta} gives that for any $x\in B_1\setminus \overline{B_{1-\rho}}$
\begin{align}\label{est eta 100}
\begin{split}
(-\Delta_g)^{s}\eta(x) 
&\leq -\text{p.v.}\, \int_{B_{1-2\rho}(-x)}a\left(\dfrac{\eta(z+x)}{(2(1-\rho))^s}  \right)\frac{dz}{(2(1-\rho))^{n+s}}\\
&\leq -\left(\dfrac{1}{2(1-\rho)} \right)^{sp+n} \int_{B_{1-2\rho}} a(\eta(n(z)))\,dz = -C\left(\dfrac{1}{2(1-\rho)} \right)^{sp+n}. 
\end{split}
\end{align}

Consider the  balls $B_{r}(x_r)$ and points $x_r$ from the interior ball condition for $\Omega$ at $x_0$. For each $r > 0$ small enough, let for any  $x \in B_r(x_r)$
$$
\alpha_r :=\dfrac{1}{C_r}\inf_{B_{r(1-\rho)}(x_r)}u,$$
and
\begin{equation}\label{defi psi}
\psi_r(x)=\dfrac{\alpha_r}{2}  d\left(\dfrac{x-x_r}{r} \right)+ \dfrac{C_r\alpha_r}{2} \eta\left(\dfrac{x-x_r}{r} \right),\end{equation}
where the constant $C_r>0$ is chosen later. Moreover, the function
$$\eta\left(\dfrac{x-x_r}{r} \right)$$has support in $B_{r(1-2\rho)}(x_r)$. Therefore,  for $x \in B_r(x_r)\setminus \overline{B_{r(1-\rho)}(x_r)}$, we obtain in the weak sense
\begin{equation}\label{choice C 3}
(-\Delta_a)^s \psi_r(x)= (-\Delta_a)^s\left( \dfrac{\alpha_r}{2}  d\left(\dfrac{\cdot-x_r}{r}\right)\right)(x)+h(x),
\end{equation}and letting $y'=(y-x_r)/r$, $x'=(x-x_r)/r$, we deduce
\begin{equation}\label{choice C 2}
\begin{split}
h(x)&\leq \dfrac{2c_1}{r^s}\int_{B_{1-2\rho}}a\left(\dfrac{C_r\alpha_r}{2r^s}\dfrac{\eta(x')-\eta(y')}{|x'-y'|^s} \right)\dfrac{dy'}{|x'-y'|^{s+n}}\\& =  \dfrac{2c_1}{r^s}\int_{B_{1-2\rho}}a\left(\dfrac{C_r\alpha_r}{2r^s}\dfrac{-\eta(y')}{|x'-y'|^s} \right)\dfrac{dy'}{|x'-y'|^{s+n}}\\& \leq -2c_1 \Phi(r)\int_{B_{1-2\rho}}a\left(\dfrac{\eta(y')}{|x'-y'|^s} \right)\dfrac{dy'}{|x'-y'|^{s+n}}= -c\Phi(r),
\end{split}
\end{equation}where
$$c=2c_1 \int_{B_{1-2\rho}}a\left(\dfrac{\eta(y')}{|x'-y'|^s} \right)\dfrac{dy'}{|x'-y'|^{s+n}}>0,$$and $\Phi$ is given by \eqref{defi phi}. Similarly, we obtain that
$$(-\Delta_a)^s\left(\dfrac{\alpha_r}{2}  d\left(\dfrac{\cdot-x_r}{r}\right)\right)= \dfrac{1}{r^s}(-\Delta_a)^s\left(\dfrac{\alpha_r}{2r^s}d\right).$$

Now, we choose the constant $C_r>1$ as follows. By Propositions \ref{prop.cont} and \ref{p a w}, there is $\lambda_0>0$ such that for all $\lambda<\lambda_0$ there holds
\begin{equation}\label{choice C}
|(-\Delta_a)^s\left(\lambda d \right)|< \dfrac{c}{2}\left(\dfrac{\inf_{B_{r/2}(x_r)}|u|}{r^s} \right)^{p-1}.
\end{equation}
Choose $C_r>1$ large enough so that
$$\dfrac{1}{C_r}\dfrac{\inf_{B_{r/2}(x_r)}|u|}{r^s}=\dfrac{\alpha_r}{r^s}<\lambda_0.$$

Therefore, recalling the definition of $\psi_r$ \eqref{defi psi} and combining \eqref{choice C 3}, \eqref{choice C 2} and \eqref{choice C} the function $\psi_r$ verifies:
\begin{align} \label{step3'}
\begin{split}
\begin{cases}
(-\Delta_a)^{s}\psi_r \leq  - \dfrac{c}{2}\Phi(r) & \text{weakly in } B_r(x_r)\setminus \overline{B_{r(1-\rho)}(x_r)},\\
\psi_r \leq \alpha_r C_r &\text{ in }B_r(x_r), \\
\psi_r= 0 &\text{ in }\mathbb{R}^{n}\setminus B_r(x_r),\\ 
\psi_r\geq  \dfrac{\alpha_r}{4r^{2}}(r^{2}-|x-x_r|^{2}) &\text{ in } B_r(x_r).
\end{cases}
\end{split}
\end{align}

\textbf{Step 3}: The function $w:= \psi_r-u^{-}$ satisfies that $u \geq w$ a.e. in $\mathbb R^n$.

\medskip

To prove this assertion we will use comparison. Observe first that $w \leq u$ in $\mathbb{R}^{n}\setminus B_r(x_r)$, moreover, in $B_{r(1-\rho)}(x_r)$ we have that
$$
w = \psi_r \leq \alpha_rC_r \leq \inf_{B_{r(1-\rho)}(x_r)}u \leq u.
$$
Also,  by \eqref{step3'},
$$(-\Delta_a)^{s} w \leq - \dfrac{c}{2}\Phi(r) + h,$$weakly in $B_r(x_r)\setminus \overline{B_{r(1-\rho)}(x_r)}$, where the function $h$ is given by
\begin{equation*}
\begin{split}
h(x)& = 2 \int_{supp\,\,u^{-}}\left[a\left(\dfrac{\psi_r(x)-\psi_r(y)+u^{-}(y)}{|x-y|^{s}}\right)-a\left(\dfrac{\psi_r(x)-\psi_r(y)}{|x-y|^{s}} \right) \right]\dfrac{dy}{|x-y|^{n+s}}\\ & =  2 \int_{supp\,\,u^{-}}\left[a\left(\dfrac{\psi_r(x)+u^{-}(y)}{|x-y|^{s}}\right)-g_p\left(\dfrac{\psi_r(x)}{|x-y|^{s}} \right) \right]\dfrac{dy}{|x-y|^{n+s}}
 \leq C^{*}
\end{split}
\end{equation*}since $\psi_r$ and $u^{-}$ are bounded and $|x-y|\geq \delta >0$, for some $\delta$. Hence, weakly in $B_r(x_r)\setminus \overline{B_{r(1-\rho)}(x_r)}$, we obtain
$$
(-\Delta_a)^{s} w \leq  C^{*} - \dfrac{c}{2}\Phi(r)   \leq 
-\|c^{-}a(u^{+})\|_{L^{\infty}(\mathbb{R}^{n})} \leq ca(u) \leq  (-\Delta_a)^{s}u, \text{ for r small enough,}
$$
where we have used \eqref{growth boundary u}. By the comparison principle, it follows $u \geq w$ in $\mathbb{R}^{n}$.

\medskip
\textbf{Step 3}: Final argument: to complete the proof, we argue as in \cite{DSV}.  For $\beta \in (0, \tfrac{\pi}{2})$, let
\begin{equation}
\mathcal{C}_\beta:=\left\lbrace x\in \Omega: \dfrac{x-x_0}{|x-x_0|}\cdot \eta > c_\beta \right\rbrace,
\end{equation}where $\eta$ is the inward normal vector joining $x_0$ with the center of the interior ball, and define the constant $c_\beta := \cos\left(\frac{\pi}{2}-\beta \right) > 0$. Take any sequence  $x_k\in \mathcal{C}_\beta$ such that $x_k \to x_0$. Then,
\begin{equation*}
\begin{split}
|x_k-x_r|^{2}& = |x_k -x_0-r\eta|^{2}= |x_k-x_0|^{2}+r^{2} -2r(x_k-x_0)\cdot \eta \\ & \leq r^{2}-|x_k-x_0|(2c_\beta r - |x_k-x_0|) < r^{2},
\end{split}
\end{equation*}for $k$ large enough. Moreover, since
$$|x_k-x_r| \geq |x_r-x_0|+|x_k-x_0| = r-|x_k-x_0| > r(1-\rho)$$for $k$ large, then $x_k \in B_r(x_r)\setminus \overline{B_{r(1-\rho)}(x_r)}$. Next,
\begin{equation*}
\begin{split}
u(x_k)& \geq w(x_k)= \psi_r(x_k) \geq \frac{\alpha_r}{r^{2}4}(r^{2}-|x_k-x_r|^{2})_{+}\\
& = \frac{\alpha_r}{4r^{2}}( r^{2}-|x_k-x_0-r\eta|^{2})_{+}\\ & = \frac{\alpha_r}{4r^{2}}(2(x_k-x_0)\cdot \eta -|x_k-x_0|^{2})_{+}\geq  \frac{\alpha_r}{4r^{2}}(2c_\beta r|x_k-x_0|-|x_k-x_0|^{2})_{+}.
\end{split}
\end{equation*}Therefore,
$$
\liminf_{k \to \infty}\dfrac{u(x_k)-u(x_0)}{|x_k-x_0|}\geq  \frac{\alpha_r}{4r^{2}}\liminf_{k \to \infty}(2c_\beta r-|x_k-x_0|)_{+} = \dfrac{\alpha_rc_\beta}{2r^{2}}.
$$
This ends the proof of the theorem. 
\end{proof}

\begin{rem}\label{remark boundary}
Observe that  although the equation is nonlocal, in the above proof we only used that $u$ is a supersolution near the boundary of $\Omega$. 
\end{rem}

By the interior sphere condition and \cite{HKS}, for any $x_0 \in \partial \Omega$, there exist $x_1\in \Omega$ and $\rho > 0$ such that
$$B_{\rho}(x_1)\subset \Omega, \quad d(x)=|x-x_0|,$$for all $x\in \Omega$ of the form
$$x= tx_1+(1-t)x_0, \quad t \in [0, 1].$$Thus,  there is a sequence $x_n \in \Omega$ such that
$$\delta(x_n)= \text{dist}\,(x_n, \partial \Omega)= |x_n-x_0| \to 0 \text{ as }n \to \infty.$$

Hence, if $u/\delta$ can be extended to a continuous function to a neighbourhood of $\partial \Omega$, Theorem \ref{teo3.intro} implies that
$$\dfrac{u(x_0)}{\delta(x_0)}
= \lim_{\overline\Omega \ni x\to x_0}\dfrac{u(x)-u(x_0)}{\delta(x)}
= \lim_{n\to \infty}\dfrac{u(x_n)-u(x_0)}{\delta(x_n)}
= \lim_{n\to \infty}\dfrac{u(x_n)-u(x_0)}{|x_n-x|}>0.$$

We establish another direct consequence. Since the function $a$ is odd, by changing $u$ by $-u$ in Theorem \ref{teo3.intro}, we obtain the following corollary.

\begin{cor}\label{corolario normal frac}
Let $\Omega \subset \mathbb{R}^n$ be open and bounded and $x_0 \in \partial \Omega$. Assume that $\Omega$ satisfies the interior ball condition at $x_0$.
Let $u$ such that $-u \in \mathcal{Z}_{x_0}$,  $u^{+}\in L^{\infty}(\mathbb{R}^{n})$, and
$$
(-\Delta_a)^{s}u \leq c(x)a(u)  \quad \text{weakly in }\Omega
$$
where $c \in L^{1}_{loc}(\Omega)$ with $c^{-}\in L^{\infty}(\Omega)$.
Further, suppose that there is $R>0$ such that $u \leq 0$ in $B_R(x_0)$, $u < 0 $ in $B_R(x_0)\cap \Omega$. Then, for every $\beta \in (0, \pi/2)$, the following strict inequality holds
\begin{equation}\label{normal derivative at boundary}
\limsup_{x\in \Omega, x\to x_0}\dfrac{u(x)-u(x_0)}{|x-x_0|} < 0,
\end{equation}whenever the angle between $x-x_0$ and the vector joining $x_0$ and the center of the interior ball is smaller than $\pi/2-\beta$. 
\end{cor}

Reasoning as before, we may also conclude, under the assumptions of Corollary \ref{corolario normal frac} and the hypothesis that $u/\delta \in C(\overline{\Omega})$, that
$$\dfrac{u(x_0)}{\delta(x_0)} < 0.$$

\section*{Acknowledgments}
P. O.  was partially supported by CONICET PIP 11220210100238CO and
ANPCyT PICT 2019-03837. P. Ochoa and A. Salort are members of CONICET.

\section*{Statements and Declarations}
\textbf{Data Availability} The author declares that the manuscript has no associated data.

\textbf{Competing interests:} The authors have no competing interests to declare that are relevant to the content of this article.

\end{document}